\documentclass{amsart}
\usepackage{graphicx}
\usepackage[utf8x]{inputenc}
\usepackage{amsmath}
\usepackage{amssymb}
\usepackage[dvipsnames]{xcolor}
\usepackage{amsthm}
\usepackage{thmtools}
\usepackage{thm-restate}
\usepackage{amsfonts}
\usepackage{graphicx}
\usepackage{stackrel}

\usepackage[hyperindex,colorlinks,breaklinks]{hyperref}
\usepackage[hyphenbreaks]{breakurl}
\usepackage[mathscr]{eucal}
\usepackage[shortlabels]{enumitem}
\usepackage{tikz}
\usepackage{tikz-cd}

\usepackage{biblatex}
\newtheorem{theorem}{Theorem}

\numberwithin{theorem}{section}
\numberwithin{equation}{section}

\newtheorem{lemma}[theorem]{Lemma}

\newtheorem{proposition}[theorem]{Proposition}

\newtheorem{question}{Question}

\newcommand{\Z}{\mathbb{Z}}
\newcommand{\pZ}{\widehat{\Z}}

\newcommand{\hD}{\widehat{\Delta}}
\newcommand{\hG}{\widehat{\Gamma}}
\newcommand{\hM}{\widehat{M}}
\newcommand{\hN}{\widehat{N}}

\newcommand{\at}[1]{|_{#1}}

\newcommand{\aug}{\mathcal{I}_\Upsilon}
\newcommand{\paug}{\mathcal{I}_{\widehat{\Upsilon}}}
\newcommand{\palg}{{\pZ[\![\widehat{\Upsilon}]\!]}}
\newcommand{\alg}{{\Z[\Upsilon]}}
\DeclareMathOperator{\Ext}{Ext}

\newcommand{\ishortcap}{%
  \scalebox{0.8}[0.63]{$\mathrm{I}$}%
}
\newcommand{\cshortcap}{%
  \scalebox{0.5}[0.63]{$\mathbf{C}$}%
}

\newcommand{\ioteda}{%
  \vcenter{\hbox{%
    \scalebox{1.1}[1.1]{%
      \ishortcap
      \kern -0.1em %
      \rule[0.46ex]{0.25em}{0.35pt} %
      \kern -0.5em %
      $\mathrm{a}$ %
}}} %
\mkern-6mu %
}

\newcommand{\iotede}{%
  \vcenter{\hbox{%
    \scalebox{1.1}[1.1]{%
      \ishortcap
      \kern -0.1em %
      \rule[0.46ex]{0.33em}{0.37pt} %
      \kern -0.6em %
      \cshortcap %
}}}}

\newcommand{\iotedesmall}{%
  \vcenter{\hbox{%
    \scalebox{1.0}[1.0]{%
      \ishortcap
      \kern -0.1em %
      \rule[0.3ex]{0.25em}{0.3pt} %
      \kern -0.54em %
      \cshortcap %
}}}}

\newcommand{\iotedasmall}{%
  \vcenter{\hbox{%
    \scalebox{1.0}[1.0]{%
      \raisebox{0ex}{\(\ishortcap\)}%
      \kern -0.1em %
      \rule[0.3ex]{0.2em}{0.3pt} %
      \kern -0.48em %
      $\mathrm{a}$}}}}

\title{The Profinite Rigidity of Free Metabelian Groups}
\author{Julian Wykowski}
\address{Department of Pure Mathematics and Mathematical Statistics, Centre for Mathematical Sciences, Wilberforce Road, CB3 0WB, United Kingdom}
\email{jw2006@cam.ac.uk}
\date{\today}

\hypersetup{
    breaklinks=true,
	linkcolor={ForestGreen},
	citecolor={ForestGreen},
	urlcolor={ForestGreen}
}

\begin{document}
\begin{abstract}
We prove that finitely generated free metabelian groups $\Psi_n$ are profinitely rigid in the absolute sense: they are distinguished by their finite quotients among all finitely generated residually finite groups. The proof is based on a previous result of the author governing profinite rigidity for modules over Noetherian domains, as well as a homological characterisation of free metabelian groups due to Groves--Miller.
\end{abstract}
\maketitle
\vspace{-0.2in}
\begin{flushright}
    \textit{Dedicated to Maria Wykowska}
\end{flushright}
\vspace{0.1in}
\section{Introduction}
Is it possible to detect the isomorphism type of a group in its finite images? This question, known as \emph{profinite rigidity}, has been a major theme in group theory for the past half-century (see \cite[Chapter 3]{Gareth_Book} for an introduction and \cite{Reid_Survey,bridson_survey2} for a survey). Usually, one works within the framework of profinite completions, enabling one to appeal to the formalism of topological groups. Indeed, the profinite completion $\widehat{\Gamma}$ of a discrete group $\Gamma$ is the profinite group given by the inverse limit of the inverse system of finite quotients of $\Gamma$. We then say that a finitely generated residually finite group $\Gamma$ is \emph{profinitely rigid} in the absolute sense if $\hG \cong \hD$ implies $\Gamma \cong \Delta$ for any finitely generated residually finite group $\Delta$. This is equivalent to the property that the isomorphism type of $\Gamma$ is distinguished by its collection of finite quotients among all finitely generated residually finite groups \cite{classic}. Despite considerable research, the absolute question has been answered for remarkably few groups: see \cite[\S 1]{Me} for an exhaustive list. In particular, the question of profinite rigidity for free and free solvable (of a fixed derived length $k \in \mathbb{N}$) groups, attributed to Remeslennikov (see \cite[Question 15]{Remeslennikov} and \cite[Question 5.48]{kourovka}), has gained significant attention (cf. the surveys \cite{Bridson_survey,bridson_survey2,Reid_Survey}). Although there exist elegant results resolving the question in certain relative settings \cite{Bridson_survey, Jaikin_Genus,Wilton_LimitGps}, the absolute case remains widely open.
\begin{question}[Remeslennikov]\label{Question}
    Are finitely generated free (free $k$-solvable) groups profinitely rigid in the absolute sense?
\end{question}
In this article, we answer Question~\ref{Question} in the case $k = 2$, i.e. for free metabelian groups. The free metabelian group of rank $n$ is the maximal metabelian (that is, solvable of derived length 2) quotient $\Psi_n$ of the free group $\Phi_n$ on $n$ generators. It is finitely generated but not finitely presentable \cite{Baumslag_Metab}.
\begin{restatable}{theoremA}{thmA}\label{MainThm}
    Finitely generated free metabelian groups are profinitely rigid in the absolute sense.
\end{restatable}
Theorem~\ref{MainThm} contrasts with \cite{Pickel_Metabelian}, where infinite families of finitely generated metabelian groups with isomorphic profinite completions were constructed, and \cite{Nikolov_Segal_2}, where uncountable such families of solvable groups were found. We invite the reader also to compare Theorem~\ref{MainThm} with \cite[Conjecture 2]{Jaikin_Genus}.

The strategy of proof is outlined as follows. First, we decompose the group $\Psi_n$ as an extension of the free abelian group $\Z^n$ by the commutator subgroup $[\Psi_n,\Psi_n]$, the latter forming an infinitely generated abelian subgroup of $\Psi_n$ which acquires naturally the structure of a finitely generated module over the group algebra $\Z[\Z^n]$. Via classical results pertaining to profinite completions, we find that any finitely generated residually finite group $\Delta$ profinitely isomorphic to $\Psi_n$ must also be of this form. A result of Groves--Miller \cite{GM} then allows us to translate the question of whether $\Delta$ is free metabelian to a statement about a certain natural extension of the augmentation ideal $\mathcal{I}_\Delta$ by the module $[\Delta,\Delta]$. The latter statement is then established via the machinery of profinite rigidity over Noetherian domains, developed by the author in \cite{Me}, in conjunction with some homological machinery for profinite completions of modules developed in Section~\ref{Sec::SepCoh} of the present article. A key step in the proof appeals to the celebrated result of Quillen and Suslin \cite{quillen, suslin} that projective modules over rings of integral polynomials are free.

\section*{Acknowledgements}
The author is grateful to his PhD supervisor, Gareth Wilkes, for enlightening discussions and advice throughout the project. Thanks are also due to Francesco Fournier-Facio for helpful conversations. Financially, the author was supported by a Cambridge International Trust \& King's College Scholarship.
\section{Preliminaries}
In this section, we outline background results regarding metabelian groups and their homological algebra, as well as the theory of profinite rigidity for modules over Noetherian domains developed by the author in \cite{Me}.

\subsection{Metabelian Groups}\label{Sec::Metab}
A discrete group $\Gamma$ is said to be \emph{metabelian} if its commutator subgroup $[\Gamma, \Gamma]$ is abelian, or equivalently, if it is solvable of derived length at most two. We invite the reader to \cite[Chapter II.6]{History} for an overview of the history and properties of these groups. Writing $\alpha_\Gamma \colon \Gamma \twoheadrightarrow \Upsilon$ for the abelianisation of a metabelian group $\Gamma$, we obtain a short exact sequence of groups
\begin{equation}\label{Eq::SESMG}
    0 \to M_\Gamma = [\Gamma, \Gamma] \to \Gamma \xrightarrow{\alpha_\Gamma} \Upsilon \to 0
\end{equation}
with abelian commutator subgroup $M_\Gamma = [\Gamma, \Gamma]$. The supergroup $\Gamma$ then acts on $M_\Gamma$ by conjugation, which factors through the projection $\alpha_\Gamma \colon \Gamma \twoheadrightarrow \Upsilon$ as $M_\Gamma$ acts trivially on itself. Thus $M_\Gamma$ acquires naturally the structure of a module over the group algebra $\alg$ whereby $\Upsilon$ acts by conjugation via any equivalent choice of set-theoretic section of the projection $\alpha_\Gamma$. Moreover, \cite[Theorem VI.6.3]{hilton_stammbach} yields the short exact sequence of $\alg$-modules
\begin{equation}\label{Eq::SESMM}
    0 \to M_\Gamma \to N_\Gamma = \alg \otimes_{\Z[\Gamma]} \mathcal{I}_\Gamma \to \aug \to 0
\end{equation}
where $\mathcal{I}_\Gamma$ and $\mathcal{I}_\Gamma$ denote the augmentation ideals of $\Gamma$ and $\Upsilon$, respectively. The boundary isomorphism in the long exact sequence associated via $\Ext_{\Z[\Gamma]}^\ast(-,-)$ to the short exact sequence of the augmentation ideal then maps the class in $H^2(\Upsilon,M_\Gamma)$ corresponding to the extension (\ref{Eq::SESMG}) to the class in $\Ext^1_{\alg}(\aug,M_\Gamma)$ corresponding to the extension (\ref{Eq::SESMM}), as per \cite[Exercise VI.6.1]{hilton_stammbach}.

The free objects in the category of metabelian groups are the \emph{free metabelian groups}. To construct the free metabelian group $\Psi_n$ on $n$ generators, begin with the free group $\Phi_n$ on $n$ generators and take the maximal metabelian quotient
\[
\Psi_n = \frac{\Phi_n}{[[\Phi_n,\Phi_n],[\Phi_n,\Phi_n]]}
\]
whose commutator subgroup has trivial commutator itself and is thus abelian. For $n > 1$, the group $\Psi_n$ is $n$-generated but not finitely presentable \cite{Baumslag_Metab}. Its abelianisation is isomorphic to the free abelian group $\Psi_n^\mathrm{ab} \cong \Z^n$ on $n$ generators. For $n = 2$, the commutator subgroup $M_{\Psi_2} = [\Psi_2,\Psi_2]$ is free as a module over the group algebra $\Z[\Psi_2^\mathrm{ab}] \cong \Z[x_1^\pm, \ldots, x_n^\pm]$ with conjugacy module structure \cite[Proposition 14.4]{km}; however, for $n > 2$, this is no longer true. Nonetheless, we obtain the following homological characterisation of free metabelian groups due to Groves--Miller \cite{GM}.

\begin{theorem}[Corollary 3 in \cite{GM}]\label{Thm::HomMet}
    Let $\Gamma$ be a finitely generated metabelian group and $\alpha \colon \Gamma \twoheadrightarrow \Upsilon$ be its abelianisation. The group $\Gamma$ is free metabelian of rank $n$ if and only if $\Upsilon \cong \Z^n$ and the $\alg$-module
    $
    N_\Gamma = \alg \otimes_{\Z[\Gamma]} \mathcal{I}_\Gamma
    $
    is free of rank $n$.
\end{theorem}
Theorem~\ref{Thm::HomMet} will be essential in the proof of profinite rigidity for free metabelian groups. To apply it, we shall make use of the theory of profinite rigidity for modules over Noetherian domains developed by the author in \cite{Me}.
\subsection{Profinite Rigidity over Noetherian Domains}
Akin to the profinite rigidity of groups one may also ask the question of profinite rigidity for modules over a Noetherian domain $\Lambda$: to what extent are these objects determined by their finite quotients? Given a $\Lambda$-module $M$, one defines the $\Lambda$-profinite completion $\hM$ as the topological $\Lambda$-modules given by the inverse limit of the inverse system of $\Lambda$-epimorphic images of $M$. We then say that a $\Lambda$-module $M$ is $\Lambda$-\emph{profinitely rigid} if $\hM \cong \hN$ implies $M \cong N$ for any finitely generated residually finite $\Lambda$-module $N$. We refer the reader to \cite[Chapter 5]{RZ} for details on profinite modules and \cite{Me} for details on profinite rigidity over Noetherian domains. We shall work exclusively in settings where $\Lambda$ is finitely generated (as a $\Z$-algebra), in which case all finitely generated modules are automatically residually finite \cite[Lemma 2.1]{Me}. The following result plays a central role in the proof of Theorem~\ref{MainThm}. We say that a ring $\Lambda$ is \emph{homologically taut} if all finitely generated projective $\Lambda$-modules are free over $\Lambda$.

\begin{theorem}[Theorem C in \cite{Me}]\label{Thm::Epic}
    Let $\Lambda$ be a homologically taut finitely generated Noetherian domain. Finitely generated free $\Lambda$-modules are $\Lambda$-profinitely rigid.
\end{theorem}

The question whether the polynomial ring $\Z[x_1, \ldots, x_n]$ is homologically taut was asked by Serre in 1955 and resolved in 1976 independently by Quillen \cite{quillen} and Suslin \cite{suslin}. We shall require the following localised formulation of this celebrated result, presently known as the Quillen--Suslin theorem.

\begin{theorem}[Corollary 7.4 in \cite{suslin}]\label{Thm::Suslin}
    The domain of integral Laurent polynomials in $n$ variables $\Z[x_1^\pm, \ldots, x_n^\pm]$ is homologically taut for any $n \in \mathbb{N}$.
\end{theorem}

\section{Homological Separability}\label{Sec::SepCoh}
In this section, we develop some of the homological methods necessary for the proof of Theorem~\ref{MainThm}. We shall require the machinery of homological algebra for profinite modules; we refer the reader to \cite[Chapter 6]{RZ} or \cite[Chapter 6]{Gareth_Book} for details on the homology and cohomology of profinite groups. Recall that a group $\Upsilon$ \emph{has separable cohomology} or \emph{is cohomologically good} if the profinite completion map $\iota \colon \Upsilon \to \widehat{\Upsilon}$ induces an isomorphism on cohomology
\[
\iota^\ast \colon H^\ast(\widehat \Upsilon, M) \to H^\ast(\Upsilon,M)
\]
whenever $M$ is a finite $\Upsilon$-module. Finitely generated free abelian groups have separable cohomology by the conjunction of \cite[Proposition 7.2.3]{Gareth_Book} and \cite[Proposition 7.3.6]{Gareth_Book}. More generally, we have the following equivalent characterisation of separable cohomology for a group of type $\operatorname{FP}_\infty$.

\begin{proposition}[Proposition 3.14 in \cite{Wilkes_SF}]\label{Prop::CohSep}
    A group $\Upsilon$ of type $\operatorname{FP}_\infty$ has separable cohomology if and only if any free resolution of the trivial $\alg$-module $\Z$
    \[
     \begin{tikzcd}
        \ldots \arrow[r] & {\alg^{k_2}} \arrow[r, "d_2"] & {\alg^{k_1}} \arrow[r, "d_1"] & {\alg} \arrow[r, "d_0"] & \Z \arrow[r] & 0
    \end{tikzcd}
    \]
    induces an exact sequence of profinite $\palg$-modules
    \[
    \begin{tikzcd}
        \ldots \arrow[r] & {\palg^{k_2}} \arrow[r, "d_2"] & {\palg^{k_1}} \arrow[r, "d_1"] & {\palg} \arrow[r, "d_0"]& \pZ \arrow[r] & 0
    \end{tikzcd}
    \]
    via the profinite completion functor in the category of $\alg$-modules. 
\end{proposition}

Typically, one defines the cohomology of a profinite group $G$ only with coefficients in discrete modules, as the category of topological modules may not have enough injectives. However, for a profinite group $G$ of type $\operatorname{FP}_\infty$, the trivial $\pZ$-module admits a resolution by finitely generated free modules, and hence it is indeed possible to define a cohomology theory of $G$ with coefficients in profinite modules \cite[pp. 379]{Symonds_Weigel}. Given a discrete group $\Upsilon$ of type $\operatorname{FP}_\infty$ with separable cohomology, its profinite completion $\widehat{\Upsilon}$ must be of type $\operatorname{FP}_\infty$ (as a profinite group) by \cite[Corollary 7.3.13]{Gareth_Book}. The following result then allows us to relate the cohomology of a group $\Upsilon$ of type $\operatorname{FP}_\infty$ with separable cohomology to the cohomology of the profinite completion $\widehat{\Upsilon}$.

\begin{proposition}\label{Prop::SepCoh}
    Let $\Upsilon$ be a group of type $\operatorname{FP}_\infty$ with separable cohomology and let $\iota \colon \Upsilon \to \widehat{\Upsilon}$ be its profinite completion. For any $\Z[\Upsilon]$-module $M$, there exists a homomorphism of graded abelian groups
    \[
    \ioteda^\ast \colon H^\ast(\Upsilon,M) \to H^\ast(\widehat{\Upsilon}, \hM)
    \]
    fitting into a commutative diagram
    \begin{equation}\label{Eq::MapOnCohCom}
    \begin{tikzcd}
{H^\ast(\Upsilon,M)} \arrow[r, "\iotedasmall^\ast"] \arrow[d, "\varpi_\ast", two heads] & {H^\ast(\widehat{\Upsilon},\widehat{M})} \arrow[d, "\varpi_\ast", two heads] \\
{H^\ast(\Upsilon,Q)}                                                   & {H^\ast(\widehat\Upsilon,Q)} \arrow[l, "\iota^\ast"']              
\end{tikzcd}
     \end{equation}
    whenever $\varpi \colon M \twoheadrightarrow Q$ is an epimorphism of $\Z[\Upsilon]$-modules and $Q$ is finite.
\end{proposition}

\begin{proof} Let $\Upsilon$ be a group of type $\operatorname{FP}_\infty$ with separable cohomology and choose a finitely generated free resolution
   \begin{equation}\label{Eq::Resolution}
       \begin{tikzcd}
        \ldots \arrow[r] & {\Z[\Upsilon]^{k_2}} \arrow[r, "d_2"] & {\Z[\Upsilon]^{k_1}} \arrow[r, "d_1"] & {\Z[\Upsilon]} \arrow[r, "d_0"] & \Z \arrow[r] & 0
    \end{tikzcd}
   \end{equation}
   of the trivial $\Z[\Upsilon]$-module $\Z$. By \cite[Proposition 3.1]{Jaikin1}, the resolution (\ref{Eq::Resolution}) induces an exact sequence of profinite $\Z[\Upsilon]$-modules via the profinite completion map, so we obtain a commutative diagram 
   \begin{equation}\label{Eq::CommutativeResolutions1}
   \begin{tikzcd}
\ldots \arrow[r] & {\pZ[\![\widehat\Upsilon]\!]^{k_2}} \arrow[r, "\widehat{d_2}"] & {\pZ[\![\widehat\Upsilon]\!]^{k_1}} \arrow[r, "\widehat{d_1}"] & {\pZ[\![\widehat\Upsilon]\!]} \arrow[r, "\widehat{d_0}"] & \pZ \arrow[r]          & 0 \\
\ldots \arrow[r] & {\Z[\Upsilon]^{k_2}} \arrow[r, "d_2"] \arrow[u]            & {\Z[\Upsilon]^{k_1}} \arrow[r, "d_1"] \arrow[u]            & {\Z[\Upsilon]} \arrow[r, "d_0"] \arrow[u]            & \Z \arrow[r] \arrow[u] & 0
\end{tikzcd}
   \end{equation}
  whose rows are exact and whose vertical arrows agree with the canonical profinite completion morphisms of the free $\Z[\Upsilon]$-modules. Now, suppose that $M$ is a $\Z[\Upsilon]$-module. By the universal property of profinite completions, any morphism of $\Z[\Upsilon]$-modules $\zeta \colon \Z[\Upsilon]^{k_i} \to M$ induces a continuous morphism of $\Z[\Upsilon]$-profinite completions $\widehat\zeta \colon \pZ[\![\widehat\Upsilon]\!]^{k_i} \to \hM$ making the diagram
  \[
    \begin{tikzcd}
        {\pZ[\![\widehat{\Upsilon}]\!]^{k_i}} \arrow[r, "\widehat\zeta"]                & \hM                    \\
        {\Z[\Upsilon]^{k_i}} \arrow[u, "{\iota_{\Z[\Upsilon]}}"] \arrow[r, "\zeta"] & M \arrow[u, "\iota_M"]
    \end{tikzcd}
    \]
    commute. Hence we obtain a homomorphism
   \[
   \chi^\bullet_M \colon \operatorname{Hom}_{\Z[\Upsilon]}(\Z[\Upsilon]^\bullet, M) \to \operatorname{Hom}_{\pZ[\![\widehat\Upsilon]\!]}(\pZ[\![\widehat{\Upsilon}]\!]^\bullet, \widehat{M})
   \]
    which forms a chain map owing to the commutativity of (\ref{Eq::CommutativeResolutions1}). Thus $\chi^\bullet$ induces a map on cohomology
   \[
   \ioteda^\ast \colon H^\ast(\Upsilon, M) \to H^\ast(\widehat{\Upsilon}, \widehat M)
   \] which forms a morphism of graded abelian groups. Moreover, given a finite $\Z[\Upsilon]$-epimorphic image $\varpi \colon M \to Q$ of the $\Z[\Upsilon]$-module $M$, the profinite completion map $\iota \colon \Upsilon \to \widehat{\Upsilon}$ induces an isomorphism on cohomology
   \[
   \iota^\ast \colon H^\ast(\widehat \Upsilon, Q) \to H^\ast(\Upsilon, Q)
   \]
   via the assumption that $\Upsilon$ is cohomologically separable. We then find that for any class $[\zeta] \in H^*(\Upsilon,M)$, the equation
   \[
   \iota^\ast \varpi_\ast \ioteda^\ast([\zeta]) = [\varpi \circ \widehat\zeta \circ \iota] = \varpi_\ast[\zeta]
   \] holds. We conclude that the composition $\iota^\ast \varpi_\ast \ioteda^\ast$ is equal to the induced map $\varpi_\ast$ and the diagram (\ref{Eq::MapOnCohCom}) commutes. 
\end{proof}
In dimension two, the cohomology group $H^2(\Upsilon, M)$ is known to classify extensions of $\Upsilon$ by $M$, see e.g. \cite[Section 5.5]{Gareth_Book}. The analogous statement for a profinite group $G$ remains true whenever $H^2(G,M)$ can be suitably defined, such as when $M$ is a finite module (cf. \cite[Section 5.5]{Gareth_Book}). We shall establish this correspondence also for profinite modules $M$, in the case where $G$ is the profinite completion of a discrete group of type $\operatorname{FP}_\infty$ with separable cohomology. In that case, the correspondence commutes naturally with the map $\ioteda$ constructed in Prop~\ref{Prop::SepCoh}.

Given a profinite group $G$ and a profinite $\Z[\![G]\!]$-module $A$, we shall condsider the set $\operatorname{Ext}_\mathrm{ProfGrp}(G,A)$ given by extensions of profinite groups
\[
1 \to A \xrightarrow{j} E \xrightarrow{p} G \to 1
\]
modulo the equivalence relation whereby extensions $E_1$ and $E_2$ are equivalent whenever there is an isomorphism of profinite groups $\varphi \colon E_1 \xrightarrow{\sim} E_2$ such that the induced map $\varphi j_1 \colon A \to E_2$ factors through the inclusion $j_2 \colon A \to E_2$ and the induced map $p_2 \varphi \colon  E_1 \to G$ factors through the quotient $p_1 \colon E_1 \twoheadrightarrow G$. We then obtain:

\begin{lemma}\label{Lem::Extensions}
    Let $\Upsilon$ be a group of type $\operatorname{FP}_\infty$ with separable cohomology and $\widehat{\Upsilon}$ be its profinite completion. There exists a bijective correspondence
    \[
    H^2(\widehat\Upsilon, A) \longleftrightarrow \operatorname{Ext}_{\mathrm{ProfGrp}}(\widehat \Upsilon, A)
    \]
    whenever $A$ is a finitely generated profinite $\pZ[\![\widehat\Upsilon]\!]$-module. Moreover, if $M$ is a finitely generated discrete $\Z[\Upsilon]$-module and $[\zeta] \in H^2(\Upsilon,M)$ corresponds to an extension of groups 
    $
    1 \to M \to \Gamma \to \Upsilon \to 1
    $
    then $\ioteda^2([\zeta])$ corresponds to the extension of profinite groups
    $
    1 \to \overline{M} \to \widehat \Gamma \to \widehat \Upsilon \to 1
    $
    induced via profinite completion.
\end{lemma}
\begin{proof}
    Let $\Upsilon$ be a group of type $\operatorname{FP}_\infty$ with separable cohomology and $A$ be a profinite $\pZ[\![\widehat\Upsilon]\!]$-module. By \cite[Corollary 7.3.13]{Gareth_Book}, the profinite completion $\widehat{\Upsilon}$ has type $\operatorname{FP}_\infty$ as a profinite group. Furthermore, \cite[Theorem 6.2.4]{RZ} says that the cohomology groups of $\widehat \Upsilon$ with coefficients in the topological module $A$ coincide with the homology groups of homogeneous cochains that are continuous with respect to the profinite topology on $\widehat \Upsilon$ and $A$. It then follows via the conjunction of \cite[Theorem 5.3]{TopExt} and \cite[Proposition 2.2.2]{RZ} that there exists a bijective correspondence
    \[
    H^2(\widehat\Upsilon, A) \longleftrightarrow \operatorname{Ext}_{\mathrm{ProfGrp}}(\widehat \Upsilon, A)
    \]
    as postulated. If $M$ is a finitely generated $\Z[\Upsilon]$-module and $[\zeta] \in H^2(\Upsilon,M)$ is a class corresponding to an extension of groups 
    $
    1 \to M \to \Gamma \to \Upsilon \to 1
    $
    then the profinite completion functor induces an extension of profinite groups
    \begin{equation}\label{Eq::ProfExt}
    1 \to \overline{M} \to \widehat \Gamma \to \widehat{\Upsilon} \to 1
    \end{equation}
    which is exact by \cite[Theorem 1.3.17]{Gareth_Book}. The closed abelian subgroup $\overline{M} \trianglelefteq \widehat \Gamma$ then acquires the structure of a profinite $\alg$-module via $\widehat{\Gamma}$-conjugacy, and by \cite[Lemma 5.1]{Me}, this structure is isomorphic to the profinite $\alg$-module given by the $\Z[\Upsilon]$-profinite completion $\hM$ of $M$. Hence the extension (\ref{Eq::ProfExt}) corresponds to a class $[z] \in H^2(\widehat{\Upsilon},\widehat{M})$ such that $z$ restricts to an element of $[\zeta]$ on the dense submodule of the finitely generated free module appearing in dimension two of the resolution (\ref{Eq::CommutativeResolutions1}). Proposition~\ref{Prop::SepCoh} then yields $\ioteda^2([\zeta]) = [z]$, as postulated.
\end{proof}

Similarly, given a profinite $\pZ[\![G]\!]$-module $A$ of type $\operatorname{FP}_\infty$, one may define the appropriate derived functor $\operatorname{Ext}_{\pZ[\![G]\!]}^\ast(\paug,-)$, where $\paug$ is the augmentation ideal of the profinite group $\widehat \Upsilon$ and the second coordinate takes inputs in the category of profinite $\palg$-modules, see \cite[pp. 377]{Symonds_Weigel}. Since the latter category is abelian \cite[Theorem 6.1.2]{Gareth_Book}, the group $\operatorname{Ext}^1_{\pZ[\![G]\!]}(A,B)$ classifies extensions of $B$ by $A$ in the category of profinite ${\pZ[\![G]\!]}$-modules \cite{Yoneda}. In the case where $G = \widehat{\Upsilon}$, we obtain the following result connecting the discrete and profinite extension groups. Implicit in its statement is the fact that $\widehat\aug \cong \paug$, as per \cite[Lemma 6.3.2]{RZ}.

\begin{proposition}\label{Prop::Big}
    Let $\Upsilon$ be a group of type $\operatorname{FP}_\infty$ with separable cohomology and $\widehat{\Upsilon}$ be its profinite completion. Write $\aug$ and $\paug$ to denote the augmentation ideals of $\Upsilon$ and $\widehat{\Upsilon}$, respectively. For any $\Z[\Upsilon]$-module $M$, there exists a homomorphism of graded abelian groups
    \[
    \iotede^\ast \colon \Ext^\ast_\alg(\aug,M) \to \Ext^\ast_\palg(\paug,\hM)
    \]
   fitting into the commutative diagram
   \begin{equation}\label{Eq::Naturality}
    \begin{tikzcd}
{\Ext^\ast_\palg(\paug, \hM)} \arrow[r, "\sim"] & {H^{\ast+1}(\widehat{\Upsilon}, \hM)} \\
{\Ext^\ast_\alg(\aug, M)} \arrow[u, "\iotedesmall"] \arrow[r, "\sim"]        & {H^{\ast+1}(\Upsilon, M)} \arrow[u, "\iotedasmall"]
\end{tikzcd}
 \end{equation}
    where the horizontal isomorphisms are the boundary maps in the cohomology long exact sequence associated to the short exact sequence of the augmentation ideal. Moreover, if $[\zeta] \in \Ext^1(\aug, M)$ corresponds to an extension $E$ of the $\alg$-module $\aug$ by the $\alg$-module $M$, then $\iotede^1([\zeta])$ corresponds to the extension $\widehat{E}$ of the profinite $\palg$-module $\paug$ by the profinite $\palg$-module $\hM$ induced by the $\alg$-profinite completion functor.
\end{proposition}
\begin{proof}
    Consider again the the resolution (\ref{Eq::Resolution}), which we may cap at the kernel of the map $d_0$ to obtain a free resolution of the augmentaion ideal $\mathcal{I}_\Upsilon$ of $\Upsilon$ of the form
   \begin{equation}\label{Eq::Resolution2}
       \begin{tikzcd}
        \ldots \arrow[r] & {\alg^{k_2}} \arrow[r, "d_2"] & {\alg^{k_1}} \arrow[r, "d_1"] & {\aug} \arrow[r] & 0
    \end{tikzcd}
   \end{equation}
   where the last map is a surjection by exactness of (\ref{Eq::Resolution}). Proceeding exactly as in the proof of Proposition~\ref{Prop::SepCoh}, we obtain a homomorphism
   \[
   \chi^\bullet \colon \operatorname{Hom}_{\alg}(\alg^\bullet, M) \to \operatorname{Hom}_{\pZ[\![\widehat\Upsilon]\!]}(\pZ[\![\widehat{\Upsilon}]\!]^\bullet, \widehat{M})
   \]
    whose restriction to the capped resolution (\ref{Eq::Resolution2}) yields a map on cohomology
   \[
   \iotede^\ast \colon \Ext^\bullet_{\Z[\Upsilon]}(\aug, M) \to \Ext^\bullet_{\palg}(\paug,\hM)
   \]
   which forms a morphism of graded abelian groups. As $\ioteda^\ast$ and $\iotede^\ast$ are constructed naturally with respect to the short exact sequence of $\alg$-modules $0 \to \aug \to \alg \to \Z \to 0$, the diagram (\ref{Eq::Naturality}) must commute.
   
   For the second statement, choose an extension $E$ of $\alg$-modules and and let $[\zeta] \in \Ext^1_{\alg}(\aug, M)$ be the corresponding extension class, so that $\zeta$ may be realised as a restriction $\zeta = \varphi\at{M}$ where $\varphi$ is a choice of projective lift fitting into the commutative diagram
   \[
\begin{tikzcd}
0 \arrow[r] & \operatorname{Ker}(d_1) \arrow[r] \arrow[d, dotted, "\zeta"] & \alg^{k_1} \arrow[r, "d_1"] \arrow[d, "\varphi", dotted] & \aug \arrow[d, equal] \arrow[r] & 0 \\
0 \arrow[r] & M \arrow[r]                                         & E \arrow[r]                                              & \aug \arrow[r]                                & 0
\end{tikzcd}
    \]
    as per \cite[pp. 91]{hilton_stammbach}. On the other hand, the extension $E$ induces an extension of profinite $\alg$-modules
    \begin{equation}\label{Eq::ExtProf}
        0 \to \hM \to \widehat E \to \paug \to 0
    \end{equation}
    which is exact by \cite[Lemma 5.1]{Me}. Equivalently, (\ref{Eq::ExtProf}) forms an extension of profinite modules over the profinite ring $\palg$, as in \cite[Lemma 2.2]{Me}. The category of profinite $\palg$-modules is abelian \cite[Theorem 6.1.2]{Gareth_Book} and the profinite augmentation ideal $\paug$ admits a resolution by finitely generated free modules given in Proposition~\ref{Prop::CohSep}. Consequently, the extension (\ref{Eq::ExtProf}) corresponds to a class $[z] \in \Ext^1_\palg(\paug, \hM)$ , where $z$ may be realised as a restriction $z = f\at{\hM}$ for some choice of continuous projective lift $f$ fitting into the commutative diagram
   \[
\begin{tikzcd}
0 \arrow[r] & \operatorname{Ker}(\widehat{d_1}) \arrow[r] \arrow[d, dotted, "z"] & \palg^{k_1} \arrow[r, "\widehat{d_1}"] \arrow[d, "f", dotted] & \paug \arrow[d, equal] \arrow[r] & 0 \\
0 \arrow[r] & \hM \arrow[r]                                         & \widehat E \arrow[r]                                              & \paug \arrow[r]                                & 0
\end{tikzcd}
    \]
    as per \cite{Yoneda} or \cite[Section 13.27]{stacks}. Now $z$ restricts to an element of the class $[\zeta]$ on the dense discrete submodule $\operatorname{Ker}(d_1)$ of $\operatorname{Ker}(\widehat{d_1})$, as does any choice of representative of $\iotede^1([\zeta])$. Hence $\iotede^1([\zeta]) = [z]$ and the proof is complete.
\end{proof}
Using Proposition~\ref{Prop::Big}, we are able at last to relate the $\alg$-profinite completion of the module $N_\Gamma = \alg \otimes_\alg \aug$ associated to a metabelian group (cf. Section~\ref{Sec::Metab}) to its analogue in the profinite category.
\begin{proposition}\label{Prop::Iso}
    Let $\Gamma$ be a finitely generated metabelian group and $\Upsilon$ its abelianisation. There is an isomorphism
    \[
    \widehat{N_\Gamma} = \widehat{\alg \otimes_{\Z[\Gamma]} \mathcal{I}_\Gamma} \cong \palg \widehat{\otimes}_{\pZ[\![\hG]\!]} \mathcal{I}_{\hG} = N_{\hG}
    \] as profinite $\palg$-modules.
\end{proposition}
\begin{proof}
Let $\Gamma$ be a finitely generated metabelian group and $\Upsilon$ its abelianisation. Consider the short exact sequence of groups (\ref{Eq::SESMG}) and let $[\zeta] \in H^2(\Upsilon,M_\Gamma)$ be the corresponding class. The boundary homomorphism in the long exact sequence induced by the derived functor $\Ext$ from the short exact sequence of the augmentation ideal $\aug$ maps the class $[\zeta]$ to the class $[\vartheta] \in \Ext^1(\aug,M)$ corresponding to the extension of $\alg$-modules (\ref{Eq::SESMM}), as per \cite[Exercise VI.6.1]{hilton_stammbach}. By Lemma~\ref{Lem::Extensions}, the class $\ioteda^2([\zeta])$ then corresponds to the extension of profinite groups
\begin{equation}\label{Eq::SESGpLem}
0 \to \overline{M_\Gamma} \to \hG \to \widehat \Upsilon \to 0
\end{equation}
induced by the profinite completion functor of groups. Moreover, the closure $\overline{M_\Gamma}$ of $M_\Gamma$ in $\hG$ acquires the structure of a profinite $\palg$-module via conjugacy in $\hG$, and this structure agrees with the $\alg$-profinite completion $\widehat{M_\Gamma}$ of $M_\Gamma$ by \cite[Lemma 5.1]{Me}. On the other hand, Proposition~\ref{Prop::Big} says that the class $\iotede^1([\vartheta])$ corresponds to the extension of profinite $\alg$-modules
\begin{equation}\label{Eq::SESGpMod}
0 \to \widehat{M_\Gamma} \to \widehat{N_\Gamma} \to \paug \to 0
\end{equation}
induced by the profinite completion functor of $\alg$-modules. Now the boundary homomorphism in the long exact sequence of profinite modules induced by the derived functor $\Ext$ from the short exact sequence of the augmentation ideal $\paug$ maps the class $\ioteda^2([\zeta])$ to the class in $\Ext^1_\palg(\paug, \widehat{M_\Gamma})$ corresponding to the extension of profinite $\palg$-modules
\begin{equation}\label{Eq::SESExtMod}
0 \to \widehat{M_\Gamma} \to N_{\hG} = \palg \widehat \otimes_{\pZ[\![\hG]\!]} \paug \to \paug \to 0
\end{equation}
as per \cite[Exercise VI.6.1]{hilton_stammbach}. But Proposition~\ref{Prop::Big} states that the image of $\ioteda^2([\zeta])$ under the boundary homomorphism is precisely $\iotede^1([\vartheta])$. Hence the extensions (\ref{Eq::SESGpMod}) and (\ref{Eq::SESExtMod}) are equivalent, and we obtain the postulated isomorphism $\widehat{N_\Gamma} \cong N_{\hG}$ as profinite $\palg$-modules.
\end{proof}

\section{The Proof of Theorem~\ref{MainThm}}\label{Sec::Main}
We prove that finitely generated free metabelian groups are profinitely rigid in the absolute sense. Let $\Gamma = \Psi_n$ be the free metabelian group of rank $n$ and denote by $\alpha_\Gamma \colon \Gamma \twoheadrightarrow \Upsilon \cong \Z^n$ its abelianisation. Let $\Delta$ be a finitely generated residually finite group admitting an isomorphism of profinite completions ${f \colon \hD \xrightarrow{\sim} \hG}$ as profinite groups. The abelianisation of $\Delta$ is then isomorphic to $\Upsilon \cong \Z^n$ by \cite[Proposition 3.2.10]{Gareth_Book} and the associated epimorphism $\alpha_\Delta \colon \Delta \twoheadrightarrow \Upsilon \cong \Z^n$ is unique up to multiplication by an element of $\operatorname{GL}_n(\Z)$. Write $M_\Gamma = \operatorname{Ker}(\alpha_\Gamma) \trianglelefteq \Gamma$ and $M_\Delta = \operatorname{Ker}(\alpha_\Delta) \trianglelefteq \Delta$ so that there are short exact sequences of groups
\begin{equation}\label{Eq::SESGps}
    0 \to M_\Gamma \to \Gamma \xrightarrow{\alpha_\Gamma}  \Upsilon \to 0 \qquad \text{and} \qquad 0 \to M_\Delta \to \Delta \xrightarrow{\alpha_\Delta} \Upsilon \to 0
\end{equation}
and $M_\Gamma$ is abelian since $\Gamma$ is metabelian. By the right exactness of the profinite completion functor \cite[Theorem 1.3.17]{Gareth_Book}, we then obtain two short exact sequences of profinite groups
\begin{equation}\label{Eq::SESGamma}
    0 \to \overline{M_\Gamma} \to \hG \xrightarrow{\widehat{\alpha_\Gamma}}  \widehat{\Upsilon} \to 0 \qquad \text{and} \qquad 0 \to  \overline{M_\Delta} \to \hD \xrightarrow{\widehat{\alpha_\Delta}} \widehat{\Upsilon} \to 0
\end{equation}
where $\overline{M_\Gamma}$ and $\overline{M_\Delta}$ denote the closures of the images of $M_\Gamma$ and $M_\Delta$ under the profinite completion maps $\iota_\Gamma \colon \Gamma \to \hG$ and $\iota_\Delta \colon \Delta \to \hD$, respectively. Now $\widehat{\alpha_\Gamma}$ and $\widehat{\alpha_\Delta}$ correspond to the abelianisation morphisms of the profinite groups $\hG$ and $\hD$, which are unique up to multiplication by an element of $\operatorname{GL}_n(\pZ)$. Thus we obtain a commutative diagram with exact rows
     \[
     \begin{tikzcd}
0 \arrow[r] & M_\Gamma \arrow[r] \arrow[d]                   & \Gamma \arrow[r, "\alpha_\Gamma"] \arrow[d, "\iota_\Gamma"]                             & \Upsilon \cong \Z^n \arrow[r] \arrow[d]                             & 0 \\
0 \arrow[r] & \overline{M_\Gamma} \arrow[r]                   & \widehat{\Gamma} \arrow[r, "\widehat{\alpha_\Gamma}"]                   & \widehat{\Upsilon} \cong \widehat{\Z}^n \arrow[r]                   & 0 \\
0 \arrow[r] & \overline{M_\Delta} \arrow[r] \arrow[u, dotted, "g"] & \widehat{\Delta} \arrow[r, "\widehat{\alpha_\Delta}"] \arrow[u, "f"'] & \widehat{\Upsilon} \cong \widehat{\Z}^n \arrow[r] \arrow[u, "A"'] & 0 \\
0 \arrow[r] & M_\Delta \arrow[r] \arrow[u]                   & \Delta \arrow[r, "\alpha_\Delta"] \arrow[u, "\iota_\Delta"']                                              & \Upsilon \cong \Z^n \arrow[r] \arrow[u]                             & 0
\end{tikzcd}
     \]
     for some $A \in \operatorname{GL}_n(\pZ)$. In particular, the profinite isomorphism $f$ restricts to an isomorphism of closed subgroups $g = f\at{\overline{M_\Delta}} \colon \overline{M_\Delta} \xrightarrow{\sim} \overline{M_\Gamma}$ wherefore $M_\Delta$ must be abelian as well. Now $\Gamma$ and $\Delta$ act by conjugation on $M_\Gamma$ and $M_\Delta$, respectively, and the restriction of this action to the respective abelian normal subgroups $M_\Gamma$ and $M_\Delta$ themselves must be trivial. Hence we obtain a conjugation action of $\Upsilon$ on $M_\Gamma$ and $M_\Delta$, which endows both groups with the structure of modules over the group algebra $\alg = \Z[x_1^\pm, \ldots, x_n^\pm]$. Similarly, the conjugation action within the profinite groups $\hG$ and $\hD$ endows the closed subgroups $\overline{M_\Gamma}$ and $\overline{M_\Delta}$ with a structure of modules over the completed group algebra $\palg$. The free abelian group $\Upsilon$ has separable cohomology (cf. Section~\ref{Sec::SepCoh}), so \cite[Lemma 5.1]{Me} yields the following.
    \begin{lemma}\label{Lem::ClosureVsCompletion}
         The profinite $\palg$-modules $\overline{M_\Gamma}$ and $\overline{M_\Delta}$ with $\widehat{\Upsilon}$-conjugacy module structures admit isomorphisms of $\palg$-modules $\overline{M_\Gamma} \xrightarrow{\sim}\widehat{M_\Gamma}$ and $\overline{M_\Delta} \xrightarrow{\sim}\widehat{M_\Delta}$, where $\widehat{M_\Gamma}$ and $\widehat{M_\Gamma}$ are the respective profinite completions of the $\alg$-modules $M_\Gamma$ and $M_\Delta$ with $\Upsilon$-conjugacy module structures.
     \end{lemma}
    Furthermore, the profinite isomorphism $f$ descends to the profinite automorphism $A \in \operatorname{GL}_n(\pZ)$ on abelianisation. We may extend the latter $\pZ$-linearly to an automorphism of the profinite algebra $\palg$, which we shall also denote as $A$. The restriction $g \colon \widehat{M_\Delta} \to \widehat{M_\Gamma}$ must then satisfy the equation
    \begin{equation}\label{Eq::Twistt}
        g(\omega \cdot m) = A(\omega) \cdot g(m)
    \end{equation}
    whenever $\omega \in \palg$ and $m \in \widehat{M_\Delta}$.
    A priori, a morphism of profinite abelian groups satisfying the equation (\ref{Eq::Twistt}) might not be a homomorphism of $\palg$-modules, since the conjugation action of $\pZ^n$ on $\hG$ and $\hD$ is twisted by the profinite automorphism $A$ which may not factor through an isomorphism of $\widehat{M_\Gamma}$. However, by passing to a natural extension of the module $\widehat{M_\Gamma}$ which is free over the completed group algebra $\palg$, we will be able to construct a profinite isomorphism which ``untwists'' the profinite automorphism $A$. Indeed, consider the $\alg$-modules
    \[
    N_\Gamma = \alg \otimes_{\Z[\Gamma]} \mathcal{I}_\Gamma \qquad \text{and} \qquad N_\Delta = \alg \otimes_{\Z[\Delta]} \mathcal{I}_\Delta
    \]
    which fit into the respective extensions (\ref{Eq::SESMM}) associated to the metabelian groups $\Gamma$ and $\Delta$. By Theorem~\ref{Thm::HomMet}, the $\alg$-module $N_\Gamma$ is free of rank $n$. It follows via \cite[Lemma 5.3.5]{RZ} that the profinite $\palg$-module $\widehat{N_\Gamma}$ is free profinite, i.e. $\widehat{N_\Gamma} \cong \palg^l$ for some $l \in \mathbb{N}$. On the other hand, Proposition~\ref{Prop::Iso} yields isomorphisms
    \[
     \widehat{N_\Gamma} \cong \palg \widehat{\otimes}_{\pZ[\![\hG]\!]} \mathcal{I}_{\hG} \qquad \text{and} \qquad \widehat{N_\Delta} \cong \palg \widehat{\otimes}_{\pZ[\![\hD]\!]} \mathcal{I}_{\hD}
    \]
    as profinite $\palg$-modules.
    Consider hence the homomorphism of profinite abelian groups
\[
F \colon \widehat{N_\Gamma} \cong \palg \widehat{\otimes}_{\pZ[\![\hG]\!]} \mathcal{I}_{\hG} \longrightarrow \palg \widehat{\otimes}_{\pZ[\![\hD]\!]} \mathcal{I}_{\hD} \cong \widehat{N_\Delta}
\]
given by \[
F(\omega \otimes v) = A(\omega) \otimes \widetilde{f}(v)
\]
where $\widetilde{f} \colon \mathcal{I}_{\widehat{\Gamma}} \to \mathcal{I}_{\widehat{\Delta}}$ is the map induced by $f$ on augmentation ideals. As $A$ and $\widetilde f$ are both isomorphisms of profinite $\palg$-modules, the tensor product $F$ must be an isomorphism of profinite abelian groups. Moreover, we find that
\begin{align*}  
F(\lambda \cdot \omega \otimes v) &= A(\lambda \omega) \otimes \widetilde{f}(v) \\
&= A(\lambda) \cdot A(\omega) \otimes \widetilde{f}(v)\\
&= A(\lambda) \cdot F(\omega \otimes v)
\end{align*}
holds whenever $\lambda, \omega \in \palg$ and $v \in \mathcal{I}_{\widehat{\Gamma}}$. On the other hand, the continuous automorphism $A$ of the profinite $\pZ$-algebra $\palg$ induces a continuous automorphism $\widetilde{A}$ of the product $ \palg^l \cong \widehat{N_\Gamma}$ which leaves the coordinates invariant and acts separately on each coordinate as $A$. Then $\widetilde A$ is in particular an automorphism of the underlying additive profinite groups. Hence the composition
\[
\widetilde{F} \colon \widehat{N_\Gamma} \xrightarrow{\widetilde{A}^{-1}} \widehat{N_\Gamma} \xrightarrow{\,\,F \,\,} \widehat{N_\Delta}
\]
must also be an isomorphism of profinite abelian groups. It satisfies the equation
\begin{align*}
    \widetilde{F}(\lambda \cdot \vec{\omega}) &= F\widetilde A^{-1}(\lambda \cdot \vec{\omega}) \\
    &= F\left(A^{-1}(\lambda)\cdot \widetilde A^{-1}(\vec\omega)\right)\\
    &= A\left(A^{-1}(\lambda)\right)\cdot F\left(\widetilde A^{-1}(\vec\omega)\right) \\
    &= \lambda \cdot F(\vec \omega)
\end{align*}
   whenever $\lambda \in \palg$ and $\vec \omega \in \palg^l \cong \widehat{N_\Gamma}$. Thus $\widetilde{F}$ forms a $\palg$-equivariant isomorphism of profinite abelian groups, and hence an isomorphism of profinite $\palg$-modules. After restricting scalars to the dense discrete subalgebra $\alg$ (cf. \cite[Lemma 2.2]{Me}), we may view $\widetilde{F}$ also as an isomorphism of the respective $\alg$-profinite completions of the finitely generated discrete $\alg$-modules $N_\Gamma$ and $N_\Delta$. We have proven the following.
   \begin{lemma}\label{Lem::ProfIsom}
       There is an isomorphism of $\alg$-modules $\widehat{N_\Gamma} \cong \widehat{N_\Delta}$. 
   \end{lemma}
   To complete the proof of Theorem~\ref{MainThm}, note that the group algebra of the free abelian group $\Upsilon \cong \Z^n$ is the ring of Laurent polynomials $\alg \cong \Z[x_1^\pm, \ldots, x_n^\pm]$ which forms a finitely generated Noetherian domain by Hilbert's Basis Theorem. Moreover, $\alg$ is homologically taut by Theorem~\ref{Thm::Suslin}. Consequently, we may appeal to Theorem~\ref{Thm::Epic} to find that the free $\alg$-module $N_\Gamma$ is $\alg$-profinitely rigid. The profinite isomorphism in Lemma~\ref{Lem::ProfIsom} then implies the existence of an isomorphism of finitely generated discrete $\alg$-modules $N_\Gamma \cong N_\Delta$. In particular, the $\alg$-module $N_\Delta$ is free of rank $n$. Another application of Theorem~\ref{Thm::HomMet} then shows that the metabelian group $\Delta$ must be free metabelian of rank $n$. Hence $\Delta \cong \Gamma$ and $\Gamma$ is profinitely rigid in the absolute sense. This completes the proof of Theorem~\ref{MainThm}.

\printbibliography

\end{document}